\documentclass[12pt]{extarticle}
\usepackage{amsmath, amsthm, amssymb, color}
\usepackage[colorlinks=true,linkcolor=blue,urlcolor=blue]{hyperref}
\usepackage{graphicx}
\usepackage{caption}
\usepackage{mathtools}
\usepackage{enumerate}
\usepackage{verbatim}
\usepackage{tikz,tikz-cd,tikz-3dplot}
\usepackage{amssymb}
\usetikzlibrary{matrix}
\usetikzlibrary{arrows}
\usepackage{algorithm}
\usepackage[noend]{algpseudocode}
\usepackage{caption}
\usepackage[normalem]{ulem}
\usepackage{subcaption}
\oddsidemargin \evensidemargin
\usepackage{makecell}
\usepackage{array}

\newtheorem{theorem}{Theorem}

\newtheorem{corollary}[theorem]{Corollary}

\theoremstyle{definition}

\newcommand{\PP}{\mathbb{P}}

\newcommand{\QQ}{\mathbb{Q}}
\newcommand{\CC}{\mathbb{C} }
\newcommand{\ZZ}{\mathbb{Z}}

\def\lra{{\longrightarrow}}
\def\fsl{\mathfrak{sl}}
\def\fgl{\mathfrak{gl}}

\title{\bf Four-Dimensional Lie Algebras Revisited}

\author{ Laurent Manivel, Bernd Sturmfels and Svala Sverrisd\'ottir }

\date{}
\begin{document}
\maketitle

\begin{abstract} \noindent
  The projective variety of Lie algebra structures  on a $4$-dimensional
  vector space has four irreducible components of dimension $11$.
We compute their prime ideals in the polynomial ring in $24$ variables.
By listing their degrees and Hilbert polynomials, 
we correct  an earlier publication and we
answer a 1987 question  by Kirillov and Neretin.
\end{abstract}

\section{Introduction}

We examine the projective variety ${\rm Lie}_4$ whose points are the Lie algebra structures on
the vector space $\CC^4$, with standard basis $\{e_1,e_2,e_3,e_4\}$.
Each point in ${\rm Lie}_4$ is given by a matrix
\begin{equation}
\label{eq:Amatrix}
 A \quad = \quad \begin{bmatrix}
  a_{121} & a_{131} & a_{141} & a_{231} & a_{241} & a_{341} \\
  a_{122} & a_{132} & a_{142} & a_{232} & a_{242} & a_{342} \\
  a_{123} & a_{133} & a_{143} & a_{233} & a_{243} & a_{343} \\
  a_{124} & a_{134} & a_{144} & a_{234} & a_{244} & a_{344} 
 \end{bmatrix}.
 \end{equation}
 The  $24$ matrix entries are homogeneous coordinates on $\PP^{23}$. They represent the
  structure constants of a Lie algebra, by setting
 $\,[e_i,e_j] = a_{ij1} e_1 + a_{ij2} e_2 + a_{ij3} e_3 + a_{ij4} e_4$
 for $1 \leq i \leq j \leq 4$. Since $[e_i,e_j] = -[e_j,e_i]$, we  fix the conventions
$ a_{iik} = 0$ and $a_{jik} = - a_{ijk}$. The Jacobi identity
\begin{equation}
\label{eq:Jacobi}
[\,e_i , [e_j,e_k]\,] \,\,+ \,\, [\,e_j,[e_k,e_i]\,] \,\,+\,\, [\,e_k,[e_i,e_j] \,] \quad = \quad 0 
\end{equation}
yields $16$ quadratic equations in the $24$ unknowns
$a_{ijk}$, shown explicitly in Section~\ref{sec2}.
These equations cut out the variety
 ${\rm Lie}_4 \subset \PP^{23}$.
% The following result is our primary theme.
Our discussion revolves around the following~result.

\begin{theorem} \label{thm:degrees} The variety ${\rm Lie}_4$ has four irreducible
components $C_1,C_2,C_3,C_4$ in the ambient space $\PP^{23}$. 
Each of these components has dimension $11$. Their degrees are
$55, 361, 121, 295$.
\end{theorem}

The article \cite{Man} had reported the
degrees $660, 57, 121, 195$, where the computation
used methods from intersection theory.
However, three of those numbers are incorrect.
Thus the present paper also serves as an erratum for \cite{Man}.
The corrections will be discussed in Section~\ref{sec5}.

Our study grew out of a student summer project by the third author
at MPI Leipzig. When applying the software {\tt HomotopyContinuation.jl}
\cite{BT} to the $16$ quadrics defining ${\rm Lie}_4$, she discovered numerically that the degree
equals $832$ and not $1033 = 660+57+121+195$.

We here use computer algebra to determine  the  prime ideals
for $C_1,C_2,C_3,C_4$. Their generators are
presented in Sections \ref{sec3} and \ref{sec4}.
The latter incorporates
the action of ${\rm GL}(4,\CC)$.
Theorem~\ref{thm:ideals} answer a question raised at the end of \cite{KN}
by presenting the Hilbert polynomials.

\section{Sixteen Quadrics}
\label{sec2}

Two Lie algebra structures on $\CC^4$ are isomorphic if they are in the same orbit under
the action of the general linear group $G = {\rm GL}(4,\CC)$. 
In this paper we use representation theory of the group $G$ at the level of \cite[\S 10.2]{MS}.
The matrix entries in (\ref{eq:Amatrix})
generate the polynomial ring $\CC[A] = \CC[a_{121}, a_{122}, \ldots, a_{344}]$.
We write the $G$-action  by matrix multiplication as follows:
\begin{equation}
\label{eq:action}
 G \times \CC[A] \rightarrow \CC[A], \,\,(g,A) \,\mapsto \, g^{-1} \cdot A  \cdot \wedge_2(g) . 
 \end{equation}
 The entries of the $6 \times 6$ matrix $\wedge_2(g)$ are the
 $2 \times 2$ minors of $g$, with rows and columns labeled to match (\ref{eq:Amatrix}).
 The $\ZZ^4$-grading of the polynomial ring  $\CC[A]$ induced by (\ref{eq:action}) equals
$$ {\rm deg}(a_{ijk}) \quad = \quad e_i + e_j - e_k . $$
The space of linear forms in $\CC[A]$ decomposes into two irreducible representations:
\begin{equation}
\label{eq:decomp1} 
 {\CC}[A]_1 \quad = \quad S_{(1,0,0,0)}(\CC^4) \,\, \oplus \,\, S_{(1,1,0,-1)}(\CC^4) \quad
\simeq \quad \CC^4 \,\oplus \,\CC^{20}. 
\end{equation}
Explicitly, with highest weight vectors underlined, the two $G$-invariant subspaces are
$$ \begin{small}  S_{(1,0,0,0)}(\CC^4) \, = \, \CC \bigl\{\,
\underline{a_{122}+a_{133}+a_{144}},
 -a_{121}+a_{233}+a_{244},
  -a_{131}-a_{232}+a_{344},
 -a_{141}-a_{242}-a_{343}
\bigr\}, \end{small} $$
$$ \begin{small} \begin{matrix}\!\! \!\! \!\!\!\!\!\! S_{(1,1,0,-1)}(\CC^4) \,\,= \,\,\CC \bigl\{
\, \underline{a_{124}}\,,\,
a_{123},a_{132},a_{134},a_{142},a_{143},
 a_{231},a_{234},a_{241},a_{243}, a_{341},a_{342}, \,
 a_{122} - a_{133},  \qquad \\  \,\, \,\qquad \qquad \qquad a_{133} - a_{144}, 
 a_{121} + a_{233},   a_{233} - a_{244}, 
  a_{131} - a_{232},  a_{232} + a_{344},
  a_{141} - a_{242},  a_{242} - a_{343}
\bigr\}.  \end{matrix} \qquad  \end{small} $$

The $16$ quadrics in $\CC[A]_2$ that encode the Jacobi identity (\ref{eq:Jacobi})
 are organized into a $4 \times 4$ matrix $\Theta = (\theta_{ij}) $.
 The columns  are labeled $e_4,e_3,e_2,e_1$, in this order.
The rows, labeled $123, {\bf -124}, 134, -234$,
contain the coefficients
of (\ref{eq:Jacobi}). For instance, in the second row of $\Theta$,
$$
{\bf -} \,[e_{\bf 1}, [e_{\bf 2},e_{\bf 4}]]\, -\, [e_2, [e_4,e_1]] \,-\, [e_4, [e_1,e_2]]  \quad = \quad
 \theta_{21} e_4 + \theta_{22} e_3 + \theta_{23}  e_2 + \theta_{24} e_1.
$$ 
The entries in the upper left of our $4 \times 4$ matrix $\Theta$ are
$$ \begin{small} \begin{matrix}
\theta_{11} &\!\!=\!\!& a_{124} a_{131}{-}a_{121} a_{134}{+}a_{124} a_{232}{+}a_{134} a_{233}{-}a_{122} a_{234}
                     {-}a_{133} a_{234}{+}a_{144} a_{234}{-}a_{134} a_{244}{+}a_{124} a_{344}, \\
\theta_{12} &\!\!=\!\!& a_{123} a_{131}-a_{121} a_{133}+a_{123} a_{232}-a_{122} a_{233}+a_{143} a_{234}-a_{134} a_{243}+a_{124} a_{343}, 
\qquad  \,\\
\theta_{21} &\!\!=\!\!& -\,a_{124} a_{141}+a_{121} a_{144}+a_{143} a_{234}-a_{124} a_{242}-a_{134} a_{243}+a_{122} a_{244}+a_{123} a_{344},
\ {\rm etc.}
\end{matrix} \end{small}
$$
Viewed invariantly, our matrix represents the linear map given by the following composition:
$$ \Theta \,:\,\wedge_3 \CC^4 \,\hookrightarrow \,\wedge_2 \CC^4 \,\otimes\, \CC^4 \,\rightarrow\, \CC^4 \otimes \CC^4 
\,\rightarrow\, \wedge_2 \CC^4 \,\rightarrow \,\CC^4. $$
The second and fourth map are given by Lie algebra multiplication, i.e.~by the matrix $A$.
The Jacobi identity states $\Theta = 0$. Hence ${\rm Lie}_4 \subset \PP^{23}$ is defined by the 
$16$ quadrics $\theta_{ij} \in \CC[A]_2$.

The space $\CC[A]_2 \simeq \CC^{300}$ decomposes into eight irreducible representations.
The quadrics $\theta_{ij}$ account for two of these Schur modules. After relabeling and matrix inversion,
the group $G = {\rm GL}(4,\CC)$ acts on the $ 4 \times 4$ matrix of quadrics by congruence. In symbols,
$ \Theta \, \mapsto \, g^T \Theta g $.
This action is compatible with the decomposition of $4 \times 4$ matrices into
their symmetric and skew-symmetric parts. 
Our space of quadrics $\,\CC\{ \hbox{entries of} \,\,\Theta  \,\}\,\simeq \,\CC^{16}\,$
is the direct sum of
\begin{equation}
\label{eq:tenplussix}
  \begin{matrix} S_{(1,1,1,-1)}(\CC^4) & =&  \CC \{ \hbox{entries of} \,\,\Theta + \Theta^T \} & \simeq & \CC^{10} & \quad {\rm and} \\
  S_{(1,1,0,0)}(\CC^4) & = &  \CC \{ \hbox{entries of} \,\,\Theta - \Theta^T \}  & \simeq & \CC^{6}. &
 \end{matrix} 
 \end{equation}
Highest weight vectors for these two  $G$-modules are $\theta_{11}$ and $\theta_{12} - \theta_{21}$ respectively.
The ideal of $\CC[A]$ generated by (\ref{eq:tenplussix}) is not radical.
Its radical is the intersection of four prime ideals.
These four prime ideals and their varieties in $\PP^{23}$ will be described in the next section.
Explicit lists of all ideal generators and {\tt Macaulay2} code verifying Theorems~\ref{thm:degrees} and \ref{thm:ideals}
are available at the repository website {\tt MathRepo}~\cite{mathrepo} of MPI-MiS via the link
\href{https://mathrepo.mis.mpg.de/Lie4/}{https://mathrepo.mis.mpg.de/Lie4}.

\section{Four Components}
\label{sec3}

The irreducible components of ${\rm Lie}_n$ were described by
Kirillov and Neretin \cite{KN} for $n \leq 6$ and by Carles and Diakit\'e \cite{CD} for $n=7$.
Here we revisit the basic case $n=4$. The variety ${\rm Lie}_4$ has
four irreducible components.
We denote these by $C_1,C_2,C_3,C_4$ as in  \cite[Proposition 2.1]{Man}.
This differs from the orderings used in \cite{Bas, KN}.
We begin by presenting a parametrization for each component of ${\rm Lie}_4$, starting from
a generic $G$-orbit. For this we use the matrices
\begin{equation} \label{eq:fourmatrices}
\begin{small} \begin{matrix}  A_1 \, = \,
\begin{bmatrix}
 0 & 0 & 0 & 0 &  0 & 0  \\
 1 & 0 & 0 & 0 & -2 & 0  \\
 0 & -1 & 0 & 0 &  0 & 2  \\
 0 & 0 & 0 & 1 &  0 & 0
\end{bmatrix} \, , \qquad \qquad A_2 \, = \,
\begin{bmatrix}
 0 & 0 & 0 & 0 & 0 & 0  \\
 2 & 0 & 0 & 0 & 0 & 1  \\
 0 & 1 & 1 & 0 & 0 & 0  \\
 0 & x & 1 & 0 & 0 & 0
\end{bmatrix} \, ,\medskip \\
 A_3 \, = \,
\begin{bmatrix}
0 & 0 & 0 & 0 & 0 & 0  \\
 1 & 0 & 0 & 0 & 0 & 0  \\
 0 & x & 0 & 0 & 0 & 0  \\
 0 & 0 & y & 0 & 0 & 0
\end{bmatrix} \, , \qquad
\qquad \quad A_4 \, = \,
\begin{bmatrix}
1 & 0 & 0 & 0 & 0 & 0  \\
 0 & 0 & 0 & 0 & 0 & 0  \\
 0 & 0 & 0 & 0 & 0 & 0  \\
 0 & 0 & 0 & 0 & 0 & 1
\end{bmatrix} .
\end{matrix} \end{small}
\end{equation}
The variety $C_i$ is the closure  of the $G$-orbit
of $A_i$ where $x$ and $y$ range over $\CC$. For instance,
\begin{equation}
\label{eq:C2para}
 C_2 \,\, = \,\, \hbox{closure of} \,\,\bigl\{ \,  g^{-1} \cdot A_2(x) \cdot \wedge_2(g)  \,\,: \, x\in \CC,\, g \in G\,\bigr\} \,\,
\subset \,\, \PP^{23}. 
\end{equation}
The number of free parameters is consistent with the table in
\cite[page 25]{KN}. We see there that generic $G$-orbits are dense in the varieties $C_1$ and $C_4$,
and they have codimension $1$ in $C_2$ and codimension $2$ in $C_3$.
Given our rational parametrizations, such as (\ref{eq:C2para}), we can try to find the prime ideals
$I_{C_1}, \ldots, I_{C_4}$ by performing implicitization \cite[\S 4.2]{MS} in a computer algebra system,
such as {\tt Macaulay2} \cite{M2}.
While this works in theory, it is not easy to do in practise.

\smallskip

We next recall the Lie-theoretic descriptions of the general point on each component:
\begin{itemize} 
\item[$C_1:$] The Lie algebra $\mathfrak{gl}_2$ of the general linear group
${\rm GL}(2,\CC)$. Its derived algebra is
$\mathfrak{sl}_2$.
\item[$C_2:$]  Lie algebras whose derived algebra is the Heisenberg algebra $\mathfrak{he}_3$, of dimension~$3$.
\item[$C_3:$] Lie algebras whose derived algebra is abelian and three-dimensional.
\item[$C_4:$]  The Lie algebra $2\mathfrak{aff}_2$. Its derived algebra is abelian and two-dimensional.
\end{itemize}
The derived algebra of a Lie algebra is generated by all the commutators.
As a subspace of $\CC^4$, the derived algebra is simply the image of our $4 \times 6$ matrix $A$.
This means that the $4 \times 4$-minors of $A$ vanish on all components,
while the $3 \times 3$-minors of $A$ vanish only on $C_4$.
We now present the prime ideals $I_{C_1}, I_{C_2}, I_{C_3}$ and $I_{C_4}$ in $\CC[A]$
that define the components.

\begin{theorem} \label{thm:ideals}
The ideal $I_{C_1}$ is generated by $4$ linear forms,
$10$ quadrics and $20$ cubics. The linear forms and quadrics are
$S_{(1,0,0,0)}(\CC^4)$ and $S_{(1,1,1,-1)}(\CC^4)$ in Section \ref{sec2}.
The ideal $I_{C_2}$ is generated by $16$ quadrics and $44$ cubics,
 $I_{C_3}$ is generated by $26$ quadrics and $40$ cubics, and
 $I_{C_4}$ is generated by $16$ quadrics and $60$ cubics.
The Hilbert polynomials of the varieties are
$$ \begin{matrix}
{\rm Hilb}_{C_1} & = & {\bf 55}\, [\PP^{11}]\, -\, 120 \,[\PP^{10}] \,+\, 86\, [\PP^9] \,-\, 20\, [\PP^8] , 
\qquad \qquad \qquad \qquad \qquad \qquad \\
{\rm Hilb}_{C_2} & = & \,\,{\bf 361}\, [\PP^{11}] \, -\, 1184\, [\PP^{10} ] \, +\, 1526\, [\PP^9 ] \,-\, 964  \,[\PP^8 ] \,+\, 298 \,  [\PP^7 ] \,-\, 36  \,[\PP^6 ] , \\
{\rm Hilb}_{C_3} & = & {\bf 121} \,[\PP^{11}] \, -\, 284 \,[\PP^{10}] \,+\, 220\, [\PP^9] \,-\, 56 \,[\PP^8] ,
\qquad \qquad \qquad \qquad \qquad \qquad \\
{\rm Hilb}_{C_4} & = & {\bf 295} \,[\PP^{11}] \,-\, 920\,  [\PP^{10}] \,+\, 1114  \,[\PP^9 ] \,-\, 652  \, [\PP^8 ]\, +\, 184  \, [\PP^7 ] \,-\,20 \,  [\PP^6 ].\,
\end{matrix}
$$
\end{theorem}

Here we write
$[\PP^d]$ for the Hilbert polynomial $n \mapsto \binom{n+d}{d}$ of  projective space $\PP^d$.
 The leading coefficients are the degrees of the varieties.
Theorem~\ref{thm:degrees}  is a corollary to Theorem~\ref{thm:ideals}.

\begin{proof}
We found the quadrics, cubics and quartics that vanish on each variety by
solving linear systems of equations over $\QQ$. Namely, we substituted the
parametrizations presented in (\ref{eq:fourmatrices}) into a polynomial with
unknown coefficients, and we solved for these coefficients. In order for the
linear systems to have small sizes, we organized the computations by $\ZZ^4$-degrees.

In this manner, we identified the polynomials of degree at most four in
each ideal $I_{C_i}$. We verified that the dimension of each projective variety
defined by the equations we found equals $11$. We also checked
that the degree matches that predicted  by the numerical approach with
{\tt homotopyContinuation.jl}.
The  Hilbert polynomials listed above were computed in {\tt Macaulay2} by applying the command
{\tt hilbertPolynomial} to the proposed ideal generators.

What remains to be shown is that, in each of the four cases, our polynomials do indeed generate a prime ideal.
For the ideals $I_{C_1}$ and $I_{C_3}$, this was done by simply running the  command {\tt isPrime} 
in  {\tt Macaulay2}. The computation was harder for
 $I_{C_2}$ and $I_{C_4}$. To show that~$I_{C_4}$ is generated by the 
 $26$ quadrics and $40$ cubics,
  we ran the commands {\tt \# minimalPrimes C4} and
{\tt radical C4 == C4}. Their outputs are {\tt 1} and {\tt true},
respectively. This proves the claim.

It remains to prove the claim for the variety $C_2$. We perform implicitization as follows:
\begin{small}
\begin{verbatim}
R = QQ[f1,f2,f3,f4,f5,k1,k2,k3,k4,k5,k6,m, a121,a122,a123,a124,
   a131,a132,a133,a134,a141,a142,a143,a144,a231,a232,a233,a234,
   a241,a242,a243,a244,a341,a342,a343,a344, MonomialOrder => Eliminate 12];
I = ideal(a121-f1*f4*k5-f2*f5*k5+f1*k1+f2*k3+f3*k5, a122+f4*k5-k1, a132+f4*k6-k2, 
a142+f4^2*k5+f4*f5*k6+f4*k4-f5*k2, a232+f1*f4*k6-f2*f4*k5-f1*k2+f2*k1+f4*m, 
a242+f1*f4^2*k5+f1*f4*f5*k6+f1*f4*k4-f1*f5*k2-f3*f4*k5+f4*f5*m+f3*k1,
a342+f2*f4^2*k5+f2*f4*f5*k6+f2*f4*k4-f2*f5*k2-f3*f4*k6-f4^2*m+f3*k2,
a123+f5*k5-k3,  a133+f5*k6-k4, a143+f4*f5*k5+f5^2*k6-f4*k3+f5*k1,
a131-f1*f4*k6-f2*f5*k6+f1*k2+f2*k4+f3*k6, a233+f1*f5*k6-f2*f5*k5-f1*k4+f2*k3+f5*m,
a243+f1*f4*f5*k5+f1*f5^2*k6-f1*f4*k3+f1*f5*k1-f3*f5*k5+f5^2*m+f3*k3,
a343+f2*f4*f5*k5+f2*f5^2*k6-f2*f4*k3+f2*f5*k1-f3*f5*k6-f4*f5*m+f3*k4,
a144-f4*k5-f5*k6-k1-k4, a244-f1*f4*k5-f1*f5*k6-f1*k1-f1*k4+f3*k5-f5*m,	
a234-f1*k6+f2*k5-m, a344-f2*f4*k5-f2*f5*k6-f2*k1-f2*k4+f3*k6+f4*m,
a141-f1*f4^2*k5-f1*f4*f5*k6-f2*f4*f5*k5-f2*f5^2*k6 -f1*f4*k4+f1*f5*k2+f2*f4*k3
        -f2*f5*k1+f3*f4*k5+f3*f5*k6+f3*k1+f3*k4,
a231-f1^2*f4*k6+f1*f2*f4*k5-f1*f2*f5*k6+f2^2*f5*k5+f1^2*k2-f1*f2*k1
    +f1*f2*k4+f1*f3*k6-f1*f4*m-f2^2*k3-f2*f3*k5-f2*f5*m+f3*m, a124-k5, a134-k6,
a241-f1^2*f4^2*k5-f1^2*f4*f5*k6-f1*f2*f4*f5*k5-f1*f2*f5^2*k6-f1^2*f4*k4+f1^2*f5*k2
    +f1*f2*f4*k3-f1*f2*f5*k1+2*f1*f3*f4*k5+f1*f3*f5*k6-f1*f4*f5*m+f2*f3*f5*k5
    -f2*f5^2*m+f1*f3*k4-f2*f3*k3-f3^2*k5+f3*f5*m,
a341-f1*f2*f4^2*k5-f1*f2*f4*f5*k6-f2^2*f4*f5*k5-f2^2*f5^2*k6-f1*f2*f4*k4
    +f1*f2*f5*k2+f1*f3*f4*k6+f1*f4^2*m+f2^2*f4*k3-f2^2*f5*k1+f2*f3*f4*k5
    +2*f2*f3*f5*k6+f2*f4*f5*m-f1*f3*k2+f2*f3*k1-f3^2*k6-f3*f4*m);
C2 = ideal selectInSubring( 1, gens gb(I) )
codim C2, degree C2, betti mingens C2
\end{verbatim}
\end{small}
This {\tt Macaulay2} code represents  the birational
parametrization of $C_2$ introduced in \cite[Section 3.2]{Man} and
revisited in Section \ref{sec5} below.
The ideal {\tt I} defines the graph of that parametrization.
It is prime because each of its $24$ generators is equal to 
$a_{ijk} $ minus a polynomial in the $12$ parameters.
The ideal {\tt C2} is obtained by
eliminating the $12 $ parameters $f_1,f_2,\ldots,k_6,m$, so it is prime as well.
The last line yields that it is generated by $16$ quadrics and $44$ cubics.

The parametrization is a modification of 
(\ref{eq:C2para}). Namely, our {\tt Macaulay2} code says that
$$
A = \wedge_2 g \cdot B \cdot {\rm det}(g) \cdot g^{-1} \,\,\, \,{\rm where} \,\,\,
 g \, = \, \begin{small} \begin{bmatrix}
 1 & f_1 & f_2 & f_3 \\
 0 & 1 & 0 & f_4 \\
 0 &  0 & 1 & f_5 \\
 0 &  0 & 0 & 1\end{bmatrix} \end{small}
\,\, {\rm and} \,\,\,
B \, = \, \begin{small}
\begin{bmatrix}
    0 &   0 &       0 & 0 & 0 & 0  \\
   k_1 & k_2 &       0 & 0 & 0 & 0 \\
   k_3 & k_4 &       0 & 0 & 0 & 0  \\
   k_5 & k_6 & \!k_1{+}k_4\! & m & 0 & 0 
\end{bmatrix}\! . \end{small}
$$
The five parameters $f_i$ seen in $g$
are local coordinates on the flag variety 
$ {\rm Fl}(1,3,\CC^4)$. The flag amounts to
the inclusion of the second derived algebra,
spanned by the last column of the matrix $g$, 
into the derived algebra, which is spanned by the last three columns of $g$.

The matrix $B$ depends linearly on the parameters
$k_1,\ldots,k_6,m$. It defines a rank seven
vector bundle $F$ over $ {\rm Fl}(1,3,\CC^4)$. 
The structure of $B$ ensures
that the  derived algebra is the Heisenberg algebra.
In symbols, ${\rm image}(B) \simeq \mathfrak{h}_3$.
The projective bundle $\PP(F)$ is a nonsingular variety of dimension $11 = 6+5$
that maps birationally onto $C_2$. The map is given by~{\tt I}.
  \end{proof}

We close this section by showing that the $16$ quadrics in (\ref{eq:tenplussix})
do not give a radical ideal.

\begin{corollary}
The radical ideal of $\,{\rm Lie}_4$ is minimally generated by $16$ quadrics
and $15$ quartics. The quadrics are those in (\ref{eq:tenplussix}) and
the quartics are the $4 \times 4$ minors of the matrix $A$ in (\ref{eq:Amatrix}).
\end{corollary}

\begin{proof}
The radical ideal of ${\rm Lie}_4$ is the intersection of the prime ideals
$I_{C_1} , I_{C_2}, I_{C_3}, I_{C_4}$. We computed this intersection
using {\tt Macaulay2}, and we verified that it has the asserted minimal generators.
It can also be shown directly that the ideal generated by the $16$ quadrics
in (\ref{eq:Amatrix}) is not radical. Namely, the  $4 \times 4$ minors of $A$
are not in this ideal, but their squares are.
\end{proof}

\section{Polynomials: Explicit versus Invariant}
\label{sec4}

We now examine our ideal generators through the lens
of the $G$-action. We identify the Schur modules $S_\lambda(\CC^4)$ that generate
the ideals $I_{C_i}$, and we identify polynomials that serve as highest weight vectors.
The space $\,\CC[A]_2 \simeq \CC^{300}$ decomposes into five isotypical components:
$$ \begin{matrix}
\lambda && (3,1,-1,-1) & (2,1,0,-1) & (2,0,0,0) & (1,1,1,-1) & (1,1,0,0) \\
{\rm dim} && 126 & 64 & 10 & 10 & 6 \\
{\rm mult} && 1 & 2 & 2 & 2 & 1  \\
\end{matrix}
$$
The last two columns were seen already in (\ref{eq:tenplussix}).
The space of cubics $\,\CC[A]_3 \simeq \CC^{2600}$ decomposes into
$11$ isotypical components.  We display the four components that are relevant for us:
$$ \begin{matrix}
\lambda && (2,1,1,-1) & (3,0,0,0) & (2,1,0,0) & (1,1,1,0)  && \cdots \\
{\rm dim} &&  36 & 20 & 20 & 4 && \cdots \\
{\rm mult} &&  5 & 3 & 4 & 3 && \cdots 
\end{matrix}
$$
One ingredient in an invariant description of these $G$-modules
is the {\em adjoint} ${\rm ad}(u)$ of an element $u$ in our Lie algebra.
This is the endomorphism $\CC^4 \rightarrow \CC^4$ given by $v \mapsto [u,v]$.
For any index $i \in \{1,2,3,4\}$, the adjoint of $e_i$
is represented by the $4 \times 4$  matrix $(a_{ijk})_{1 \leq j,k \leq 4}$.
 The traces of the matrices ${\rm ad}(e_i)$ are the four
linear forms that span $S_{(1,0,0,0)}(\CC^4) \subset I_{C_1}$.
To be explicit, ${\rm ad}(e_1)$ is obtained by prepending a zero column
to the left three columns of $A$. 

\smallskip

We begin with the module $S_{(3,0,0,0)}(\CC^4)$. This has
dimension $20$ and occurs with multiplicity three in the space of cubics $\CC[A]_3$.
A highest weight vector for one embedding~is
% Highest Weight Polynomial for S_{3000} in C1 and C2
%  a122^3-a122^2*a133-a122^2*a144+4*a122*a123*a132+4*a122*a124*a142-a122*
% a133^2+2*a122*a133*a144-4*a122*a134*a143-a122*a144^2+4*a123*a132*a133-4*a123*
% a132*a144+8*a123*a134*a142+8*a124*a132*a143-4*a124*a133*a142+4*a124*a142*a144+
% a133^3-a133^2*a144+4*a133*a134*a143-a133*a144^2+4*a134*a143*a144+a144^3:
$$ \begin{small} \begin{matrix} 
f_{3000} \,\,= \,\, a_{122}^3-a_{122}^2 a_{133} - a_{122}^2 a_{144} + 4 a_{122} a_{123} a_{132}
+4 a_{122} a_{124} a_{142}-a_{122} a_{133}^2 + 2 a_{122}a_{133} a_{144} \\ \qquad
-4 a_{122} a_{134} a_{143} - a_{122} a_{144}^2 + 4a_{123} a_{132} a_{133}-
4 a_{123} a_{132} a_{144} +8 a_{123} a_{134} a_{142} +8 a_{124} a_{132} a_{143}
\\ \qquad -4 a_{124} a_{133} a_{142}
 {+}4 a_{124}  a_{142} a_{144}{+} 
a_{133}^3{-}a_{133}^2 a_{144}{+}4 a_{133} a_{134} a_{143}
-a_{133} a_{144}^2{+}4 a_{134} a_{143} a_{144}{+}a_{144}^3 .
\end{matrix} \end{small}
$$
This module generates the ideal $I_{C_1}$, together
with the linear forms and quadrics seen in Theorem~\ref{thm:ideals}.
The module $G f_{3000}$ is also contained in $I_{C_2}$. But this ideal 
has $24$ additional cubic generators.
These additional cubics for $C_2$ are given by two irreducible $G$-modules:
$$ S_{(2,1,0,0)}(\CC^4) \,\oplus\, S_{(1,1,1,0)}(\CC^4) \,\, \, \simeq \,\, \, \CC^{20} \,\oplus \,\CC^4. $$
The highest weight vectors for these two irreducible $G$-modules in $\CC[A]_3$ are
$$ \begin{matrix} & {\rm trace} \bigl({\rm ad}(e_1)\cdot {\rm ad}(e_2) \cdot {\rm ad}(e_3) \bigr) 
\, - \, {\rm trace} \bigl({\rm ad}(e_2)\cdot {\rm ad}(e_1) \cdot {\rm ad}(e_3) \bigr) \smallskip \\
{\rm and} &
 {\rm trace} \bigl({\rm ad}(e_1) \bigr) \cdot {\rm trace} \bigl( {\rm ad}(e_1) \cdot {\rm ad}(e_2) \bigr)
\,-\,{\rm trace} \bigl({\rm ad}(e_2) \bigr)\cdot {\rm trace} \bigl( {\rm ad}(e_1)^2 \bigr).
\end{matrix}
$$
These are cubic polynomials in the $24$ unknowns $a_{ijk}$, having
$51$ and $39$ terms respectively.

We also consider the trace of the third power of the adjoint of $e_1$. This is the cubic
$$ \!\! \begin{matrix} g_{3000} \,=\, {\rm trace}\bigl({\rm ad}(e_1)^3\bigr) \,= \,
 a_{122}^3+a_{133}^3+a_{144}^3+3 (a_{122} a_{123} a_{132}+a_{122} a_{124} a_{142} 
 +a_{123} a_{132} a_{133}\\
 \qquad \qquad \qquad \qquad \quad
 +\,a_{123} a_{134} a_{142}+a_{124} a_{132} a_{143}+a_{124} a_{142} a_{144}+a_{133} a_{134} a_{143}+a_{134} a_{143} a_{144}).
\end{matrix}
$$
This is the highest weight vector of another embedding of $S_{(3,0,0,0)}(\CC^4)$ 
into $I_{C_1}$. This $G$-module can also serve to generate $I_{C_1}$. We note that $g_{3000} - f_{3000}$
lies in the ideal generated by the four linear forms in $I_{C_1}$.
However, unlike $f_{3000}$, the cubic $g_{3000}$ does not lie in $I_{C_2}$.

\smallskip

We now turn to the two components $C_3$ and $C_4$. Lie algebras in these
components have the property that their
derived algebra is abelian. Hence the second derived algebra is zero.

The ideal $I_{C_3}$ requires $10$ additional quadrics and $40$ cubics. 
These $10$ quadrics span a module $\,S_{(1,1,1,-1)}(\CC^4)\,$ inside $\,\CC[A]_2$, with highest weight vector
$\,a_{124} a_{344} - a_{134} a_{244} + a_{144} a_{234}$.
Hence both embeddings of $S_{(1,1,1,-1})(\CC^4)$ into the space of quadrics $\CC[A]_2$ vanish on $C_3$.

The second derived algebra of a Lie algebra is the column space of the matrix $\, A \cdot (\wedge_2 A) $.
This is the product of a $4 \times 6$ matrix with a $6 \times 15$ matrix.
The $60$ entries of the $4 \times 15$ matrix $A \cdot (\wedge_2 A) $ are cubics.
These polynomials lie in both $I_{C_3}$ and $I_{C_4}$ because the second derived algebra is zero.
Twenty of the $60$ matrix entries are already accounted for: they are in the ideal of (\ref{eq:tenplussix}).
The remaining $40$ cubics are generators of $I_{C_3}$ and they form a $G$-module
\begin{equation}
\label{eq:C3C4module}
S_{(2,1,1,-1)}(\CC^4) \,\oplus \, 
S_{(1,1,1,0)}(\CC^4) \,\,\, \simeq \,\,\,
 \,\CC^{36} \,\oplus \, \CC^4. 
 \end{equation}
 This module also gives $40$ of the minimal generators of $I_{C_4}$.
 The remaining $20$ cubic generators of $I_{C_4}$ 
 form a Schur module $ S_{(3,0,0,0)}(\CC^4)$. They arise
 from the constraint that the matrix $A$ has rank two.
 A highest weight vector is the lower left $3 \times 3$-minor of the matrix $A$, which equals
$\,  a_{122} a_{133} a_{144}-a_{122} a_{134} a_{143}
-a_{123} a_{132} a_{144}+a_{123} a_{134} a_{142}+a_{124} a_{132} a_{143} - a_{124} a_{133} a_{142}$.

In conclusion, the polynomials in a $G$-invariant ideal can be described either invariantly,
using the language of representation theory, or explicitly as linear combinations of monomials.
This section straddles between these two paradigms. We explored
the four associated prime ideals $I_{C_i}$ of ${\rm Lie}_4$.
While some readers favor the invariant description, others will prefer to work with
the explicit polynomials which we posted at
\href{https://mathrepo.mis.mpg.de/Lie4/}{https://mathrepo.mis.mpg.de/Lie4}.

\section{From Vector Bundles to Degrees}
\label{sec5}

The approach in \cite{Man} is based on desingularization
of each component $C_i$, following Basili \cite{Bas}. 
The key idea is to {\em linearize the Jacobi identity}, i.e.~to
express (\ref{eq:Jacobi}) by linear equations in $A$
over a variety derived from ${\rm GL}(4,\CC)$.
This yields alternative parametrizations
which allow for the
computation of degrees using Chern classes.
We now correctly derive the numbers in Theorem \ref{thm:degrees}
by this method. In what follows we present this for $C_2$ and then for $C_1$.
The derivation for $C_4$ is similar to that for $C_2$, and that for $C_3$
was correct in \cite[Section 3.3]{Man}.

The linearized parametrization for $C_2$ was shown in the {\tt Macaulay2} code
in the proof of Theorem~\ref{thm:ideals}.
It is given by a vector bundle $F$ of rank $7$ over the 
$5$-dimensional flag variety
${\rm Fl}(1,3,\CC^4)$.
The degree of $C_2$ in $\PP^{23}$ is computed 
using Chern classes and Segre classes:
\begin{equation}
\label{eq:degC2integrals}
{\rm deg}(C_2) \,\,=\,\,\int_{C_2}c_1(O(1))^{11}\,
=\, \int_{ \PP (F)}c_1(O_F(1))^{11}\,=\,\int_{{\rm Fl}(1,3,V)}s_5(F).
\end{equation}
From now  set $V = \CC^4$. The pair $(L,U) \in {\rm Fl}(1,3,V)$ gives
the second and first derived~algebra.
We compute the  Segre class in (\ref{eq:degC2integrals}),
as in \cite[Section 3.2]{Man}. This uses the exact sequence 
$$0\,\,\lra\,\, K\,\simeq\, {\rm Hom}(V/U, {\rm End}^0_L(U))
\,\,\lra\,\, F\,\,\lra \,\,M={\rm Hom}(\wedge^2(U/L),L) \,\,\lra\,\, 0.$$
Here, $M$ is a line bundle and $K$ is a rank six vector bundle.
Note the matching parameters $m$ and $k_1,k_2,k_3,k_4,k_5,k_6$
 in the proof of Theorem~\ref{thm:ideals}.
The kernel $K$ fits into the exact sequence 
$$0 \,\,\lra\,\, K\,\,\lra \,\, {\rm Hom}\bigl(V/U, {\rm End}(U) \bigr) \,\,\lra\,\, 
{\rm Hom}\bigl(V/U, {\rm Hom}(L,U) \bigr) \,\,\lra\,\, 0.$$
From the convention that the Segre class and Chern class are inverse to each other we get
\begin{equation}
\label{eq:sF} s(F)\,\,=\,\,s(M)\cdot s(K)\,=\,s(M) \cdot s 
\bigr({\rm Hom}(V/U, {\rm End}(U))\cdot c( {\rm Hom}(V/U, {\rm Hom}(L,U)\bigr). \end{equation}
This is a rational generating function $s(x,y)$ in $x=-c_1(L)$ and $y=-c_1(U)=c_1(V/U)$. These classes
satisfy $x^4 = y^4 = 0$ because they are induced
from $\PP(V)$ and  $\PP(V^\vee)$ respectively.

Let $s_5(x,y)$ denote the component of degree $5$ in $s(x,y)$. We need to 
integrate it on the flag variety. 
Since ${\rm Fl}(1,3,V)$ is a
 divisor of type $(1,1)$ in $\PP(V)\times \PP(V^\vee)$, we can write
$$ {\rm deg}(C_2)\,\,\,=\,\,\,\int_{F(1,3,V)} \!\! s_5(F)
\,\,\,=\,\,\,\int_{\PP(V)\times \PP(V^\vee)}\!\!\! (x+y)s_5(x,y).$$
This says that the desired degree is the
coefficient of $x^3y^3$ in the polynomial $(x+y)s_5(x,y)$.

To find $s(x,y)$ we compute the three factors on the right of (\ref{eq:sF}).
Since $M$ is a line bundle, we have $s(M)=1/(1+y-2x)$.
The second factor is pulled back from $\PP(V^\vee)$. We obtain
$$s \bigl({\rm Hom}(V/U, {\rm End}(U) )\bigr) \,\,=\,\, \frac{(1-2y)^4}{(1-y)^{17}} .  $$
A twist of the tautological sequence from $\PP(V^\vee)$ finally gives 
$$c \bigl( {\rm Hom}(V/U, {\rm Hom}(L,U)) \bigr)\,\,=\,\,\frac{(1+x-y)^4}{1+x}.  $$
Multiplying all these rational functions and extracting the coefficient of $x^3y^3$ gives {\bf 361}.

\smallskip

We next consider the first component $C_1$ and we show ${\rm degree}(C_1)=55$.
This corrects \cite[Section 3.1]{Man}. The mistake was due to the fact that the model
in \cite[Proposition 3.1]{Man} was only birational, 
with non-empty indeterminacy locus which was not properly taken into account. 

To correct this we will use a slightly different, better behaved  model.
We recall from \cite[\S 2]{Bas} or \cite[\S 2.1]{Man} that, for $U \simeq \CC^3$,
a Lie bracket $\theta\in {\rm Hom}(\wedge^2U,U)$ defining a Lie algebra structure isomorphic to $\fsl_2$
must belong to the space ${\rm Hom}_s(\wedge^2U,U)$ of symmetric matrices.
 This relies on the fact that $\wedge^2U\simeq U^\vee\otimes \det(U)$, 
 since $U$ is $3$-dimensional, and therefore
$${\rm Hom}(\wedge^2U,U)\,\,\simeq \,\,
U\otimes U\otimes \det(U)^\vee \,\,= \,\,
(S^2U\oplus\wedge^2U)\otimes \det(U)^\vee.$$
The first component $S^2U\otimes \det(U)^\vee=:{\rm Hom}_s(\wedge^2U,U)$ is a $6$-dimensional subspace of
${\rm Hom}(\wedge^2U,U) \simeq \CC^9$. Explicitly, if $u$ is a vector in $U$ and $\Omega\in\det(U)^\vee$ 
is a volume form
on $U$, then the Lie bracket defined by $u^2\otimes \Omega$ is given by the formula
$$\theta (x,y)\,=\,\Omega(x,y,u) \cdot u \qquad \hbox{for all}\,\, x,y\in U.$$
By linearity, this formula extends to any element $q=u_1^2+u_2^2+u_3^2$ of $S^2U$, and the resulting 
Lie algebra structure on $U$ is isomorphic to $\fsl_2$ exactly when $q$ is non-degenerate. 

Now suppose that $V=\CC^4$ has a Lie bracket $\theta$ such that the Lie algebra structure is isomorphic
to $\fgl_2$. The center $L$ is $1$-dimensional. We have $\theta(L\wedge V)=0$
since elements of $L$ commute with all vectors in $V$.
Therefore $\theta$ descends to an  element of 
${\rm Hom}(\wedge^2Q,V)$, where $Q=V/L$. Composing with the projection to $Q$, we get a Lie bracket 
$\overline{\theta}\in {\rm Hom}(\wedge^2Q,Q)$ which, by the previous
paragraph, must be a non-degenerate element of ${\rm Hom}_s(\wedge^2Q,Q)$. In other words, 
we can describe $\theta$ as a generic matrix in the subspace of ${\rm Hom}(\wedge^2Q,V)\subset {\rm Hom}(\wedge^2V,V)$
consisting of elements whose projection to ${\rm Hom}(\wedge^2Q,Q)$ is symmetric.

When the line $L$ varies, the construction above defines a rank nine vector bundle $E$ over $\PP(V)=\PP^3$.
This fits into the following exact sequence, where  $L$ and $Q$, with a slight abuse of notation, now denote the tautological line bundle and the quotient vector bundle:
\begin{equation}
\label{eq:tautological}
 0 \,\,\lra\,\, {\rm Hom}(\wedge^2Q,L)\,\, \lra \,\,E \,\, \lra \,\, {\rm Hom}_s(\wedge^2Q,Q)\,\,\lra \,\,0. 
 \end{equation}
Since $E$ was constructed as a subbundle of the trivial vector bundle with fiber ${\rm Hom}(\wedge^2V,V)$
there is a natural morphism from the total space of $\PP(E)$ to $\PP({\rm Hom}(\wedge^2V,V))=\PP^{23}$.
This is a resolution  of singularities of the irreducible component $C_1$. 
This allows us to compute the degree of $C_1$ in its embedding into $\PP^{23}$ exactly as before:
\begin{equation}
\label{eq:degC1integrals}
{\rm deg}(C_1) \,\,=\,\,\int_{C_1}c_1(O(1))^{11}\,
=\, \int_{ \PP (E)}c_1(O_E(1))^{11}\,=\,\int_{\PP(V)}s_3(E).
\end{equation}
The exact sequence (\ref{eq:tautological}) lets us compute the Segre class of $E$ in terms of $x=-c_1(L)$. 
We~get 
$$s(E)=\frac{1-3x}{(1-x)^{10}}=1+7x+25x^2+55x^3. $$
This establishes the desired equation $\,\deg(C_1)= {\bf 55}$.

\bigskip
\bigskip

\noindent {\bf Acknowledgments}.
We thank Marc  H\"ark\"onen and Leonid Monin
for help with this project.

\bigskip
 \bigskip

\noindent
\footnotesize
{\bf Authors' addresses:}

\smallskip

\noindent Laurent Manivel, 
Paul Sabatier University, Toulouse
\hfill {\tt laurent.manivel@math.cnrs.fr}

\noindent Bernd Sturmfels,
MPI-MiS Leipzig and UC Berkeley
\hfill {\tt bernd@mis.mpg.de}

\noindent Svala Sverrisdóttir,
UC Berkeley
\hfill {\tt svalasverris@berkeley.edu}


\bigskip

\begin{thebibliography}{10}
\begin{small}
\setlength{\itemsep}{-0.6mm}

\bibitem{Bas} R.~Basili: {\em Resolutions of singularities of varieties of
  Lie algebras of dimensions 3 and 4}, J.~Lie Theory
  {\bf 12} (2002) 397--407.

 \bibitem{BT} P.~Breiding and S.~Timme:
HomotopyContinuation.jl: A Package for Homotopy Continuation in Julia,
\emph{Math.~Software -- ICMS 2018}, 458--465, Springer, 2018.

\bibitem{CD} R.~Carles and Y.~Diakit\'e: {\em Sur les vari\'et\'es d'alg\`ebres de Lie de dimension $\leq 7$},
Journal of Algebra {\bf 91} (1984) 53--63.

\bibitem{mathrepo}
C.~Fevola and C.~G\"{o}rgen: {\em The mathematical research-data repository MathRepo},
Computeralgebra Rundbrief {\bf 70} (2022) 16--20.

\bibitem{M2} D.~Grayson  and M.~Stillman:  Macaulay2, a software system
for research in algebraic geometry, available at
{\tt http://www.math.uiuc.edu/Macaulay2/}.

\bibitem{KN}
  A.~Kirillov and Y.~Neretin: {\em The variety $A_n$ of structures of
    $n$-dimensional Lie algebras},
    American Mathematical Society Translations {\bf 137} (1987) 21--30.   

\bibitem{Man} L.~Manivel: {\em On the variety of four dimensional
Lie algebras}, J.~Lie Theory {\bf 26} (2016) 1--10.

\bibitem{MS} M.~Micha{\l}ek and B.~Sturmfels: {\em Invitation to Nonlinear Algebra},
Graduate Studies in Mathematics, vol 211, American Mathematical~Society, Providence, 2021.
\end{small}
\end{thebibliography}
\end{document}